\documentclass[11pt,reqno,tbtags]{amsart}
\usepackage{amssymb}
\usepackage{url}
\usepackage[square,numbers]{natbib}
\bibpunct[, ]{[}{]}{;}{n}{,}{,}
\title%[]
{The probability that a random multigraph is simple, II}

\date{23 July, 2013}
\author{Svante Janson}
\thanks{Partly supported by the Knut and Alice Wallenberg Foundation}
\address{Department of Mathematics, Uppsala University, PO Box 480,
SE-751~06 Uppsala, Sweden}
\email{svante.janson@math.uu.se
  \qquad http://www2.math.uu.se/{\tiny$\sim$}svante/}
\subjclass[2010]{05C80; 05C30, 60C05} 
\numberwithin{equation}{section}

\renewcommand\le{\leqslant}
\renewcommand\ge{\geqslant}

\allowdisplaybreaks

\newtheorem{theorem}{Theorem}[section]
\newtheorem{lemma}[theorem]{Lemma}

\theoremstyle{definition}

\newtheorem{remark}[theorem]{Remark}

\newtheorem*{ack}{Acknowledgement}

\theoremstyle{remark}

\newenvironment{romenumerate}[1][0pt]{% optional argument changes indentation
\addtolength{\leftmargini}{#1}\begin{enumerate}% gives (i), (ii) etc.
 }{\end{enumerate}}

\newcounter{oldenumi}
{\setcounter{oldenumi}{\value{enumi}}
\begin{romenumerate} \setcounter{enumi}{\value{oldenumi}}}
{\end{romenumerate}}

\newcounter{thmenumerate}

\newcounter{romxenumerate}   %less indented than standard.

\newcounter{xenumerate}   %no left indentation; thus wider lines

\newcommand\pfitem[1]{\par(#1):}

\newcommand{\refT}[1]{Theorem~\ref{#1}}

\newcommand{\refL}[1]{Lemma~\ref{#1}}
\newcommand{\refR}[1]{Remark~\ref{#1}}
\newcommand{\refS}[1]{Section~\ref{#1}}

\newcommand{\refand}[2]{\ref{#1} and~\ref{#2}}

\newcommand\REM[1]{{\raggedright\texttt{[#1]}\par\marginal{XXX}}}

\begingroup
  \divide\count255 by 60
  \multiply\count255 by -60
  \advance\count255 by \time
  \ifnum \count255 < 10 \xdef\klockan{\the\count1.0\the\count255}
  \else\xdef\klockan{\the\count1.\the\count255}\fi
\endgroup

\newcommand{\sumin}{\sum_{i=1}^n}

\newcommand\set[1]{\ensuremath{\{#1\}}}

\newcommand\xpar[1]{(#1)}
\newcommand\bigpar[1]{\bigl(#1\bigr)}
\newcommand\Bigpar[1]{\Bigl(#1\Bigr)}
\newcommand\biggpar[1]{\biggl(#1\biggr)}
\newcommand\lrpar[1]{\left(#1\right)}

\newcommand\xcpar[1]{\{#1\}}

\def\rompar(#1){\textup(#1\textup)}    % usage: \rompar(...)

\newcommand\parfrac[2]{\lrpar{\frac{#1}{#2}}}

\newcommand\Bigparfrac[2]{\Bigpar{\frac{#1}{#2}}}

\def\xexp(#1){e^{#1}}

\newcommand\nn{[n]}
\newcommand\ntoo{\ensuremath{{n\to\infty}}}
\newcommand\Ntoo{\ensuremath{{N\to\infty}}}

\newcommand\punkt[1]{\if.#1\else.\spacefactor1000\fi{#1}}
\newcommand\iid{i.i.d\punkt}    
\newcommand\ie{i.e\punkt}
\newcommand\eg{e.g\punkt}

\newcommand\cf{cf\punkt}

\newcommand\whp{w.h.p\punkt}

\newcommand{\tend}{\longrightarrow}
\newcommand\dto{\overset{\mathrm{d}}{\tend}}

\newcommand\eqd{\overset{\mathrm{d}}{=}}

\newcounter{CC}
\newcounter{cc}
\newcommand\E{\operatorname{\mathbb E{}}}
\renewcommand\P{\operatorname{\mathbb P{}}}

\newcommand\Po{\operatorname{Po}}

\newcommand\Be{\operatorname{Be}}

\newcommand\ga{\alpha}
\newcommand\gb{\beta}
\newcommand\gd{\delta}

\newcommand\gl{\lambda}

\newcommand\gO{\Omega}

\newcommand\eps{\varepsilon}

\renewcommand\phi{\xxx}  %% WARNING

\newcommand\cH{\mathcal H}

\newcommand\ett[1]{\boldsymbol1\xcpar{#1}}

\newcommand\qw{^{-1}}
\newcommand\qww{^{-2}}
\newcommand\qq{^{1/2}}
\newcommand\qqw{^{-1/2}}

\renewcommand{\=}{:=}

\newcommand\ooo{[0,\infty)}

\newcommand\dtv{d_{\mathrm{TV}}}

\newcommand{\mgf}{moment generating function}

\newcommand\gnd{\ensuremath{G(n,(d_i)_1^n)}}
\newcommand\gndx{\ensuremath{G^*(n,(d_i)_1^n)}}

\newcommand\ggb{\ensuremath{G\bigpar{n,(s_i)_1^{n'},(t_j)_1^{n''}}}}
\newcommand\ggbx{\ensuremath{G^*\bigpar{n,(s_i)_1^{n'},(t_j)_1^{n''}}}}
\newcommand\ggx{\ensuremath{G^*}}

\newcommand\ettoN{\bigpar{1+O(N\qw)}}
\newcommand\ettoNx{{1+O(N\qw)}}
\renewcommand\ij{_{ij}}
\newcommand\xij{^{(i,j)}}
\newcommand\xii{^{(i,i)}}
\newcommand\hX{\widehat X}
\newcommand\hY{\widehat Y}
\newcommand\hZ{\widehat Z}
\newcommand{\sumi}{\sum_{i}}
\newcommand{\sumij}{\sum_{i<j}}
\newcommand\nnn{^{(n)}}
\newcommand\dn{\ensuremath{(d_i)_1^n}}

\newcommand\hd{\bar{d}}
\newcommand\hn{\bar{n}}

\newcommand\sumd{\sum_i d_i}
\newcommand\sumdd{\sum_i d_i^2}

\newcommand\maxd{\max_i d_i}
\newcommand\gab{_{\ga\gb}}
\newcommand\gagb{_{\ga,\gb}}
\newcommand\tI{\tilde I}
\newcommand\tX{\widetilde X}
\newcommand\tY{\widetilde Y}
\newcommand\uI{J}
\newcommand\ub{\check}
\newcommand\uX{\ub X}
\newcommand\uY{\ub Y}
\newcommand\ugl{\ub\gl}
\newcommand\uc{\overline}
\newcommand\zoo{Z_\infty}
\newcommand\hG{\overline G}
\newcommand{\Holder}{H\"older}

\begin{document}

\begin{abstract} 
Consider a random multigraph $G^*$ with given vertex degrees
$d_1,\dots,d_n$,
constructed by the configuration model.
We give a new proof of the fact that, 
asymptotically for a sequence of such multigraphs with
the number of edges $\tfrac12\sumd\to\infty$,
the probability that the multigraph is simple stays away from 0 if and
only if $\sumdd=O\bigpar{\sumd}$. 
The new proof uses the method of moments, which makes it possible to 
use it in some applications concerning convergence in distribution.

Correponding results for bipartite graphs are included.
\end{abstract}

\maketitle

\section{Introduction}\label{S:intro}

Let
\gnd{} be the random (simple) graph with 
vertex set $\nn\=\set{1,\dots,n}$
and vertex degrees $d_1,\dots,d_n$, uniformly chosen among all
such graphs. (We assume that there are any such graphs at all; in particular,
$\sum_i d_i$ has to be even.) 
A standard method to study $\gnd$ is to consider the related random
labelled multigraph \gndx{} defined by taking a set of $d_i$ \emph{half-edges}
at each vertex $i$ and then joining the 
half-edges into edges by taking a random partition of the set of all
half-edges into pairs.
This is known as the configuration model, 
and was introduced by 
Bollob\'as \cite{BB80},
see also  \cite[Section II.4]{Bollobas}.
(See Bender and Canfield
\cite{BenderC} and Wormald \cite{WormaldPhD,Wormald81} for related
arguments.)
Note that \gndx{} is defined for all $n\ge1$ and all sequences $\dn$
such that $\sumd$ is even (we tacitly assume this throughout the
paper), 
and that we obtain \gnd{} if we condition \gndx{} on being a simple
graph. 

It is then important to estimate the probability 
that \gndx{} is simple, and in particular to decide whether 
\begin{equation}
\label{a2}
\liminf_{\ntoo}\P\bigpar{\gndx\text{ is simple}}>0  
\end{equation}
for  given sequences $\dn=(d_i\nnn)_1^n$. 
(We assume throughout that we consider a sequence of instances, and consider
asymptotics as \ntoo. Thus our
degree sequence $\dn$ depends on $n$, 
and so do other quantities introduced below;
for simplicity, we omit this from the notation.)
Note that \eqref{a2} implies that any statement holding for \gndx{}
with probability tending to 1 as \ntoo{} does so for \gnd{} too.
(However, note also that
\citet{BollobasRiordan} have recently shown that the method may be
applied even when \eqref{a2} does not hold; in the problem they study,
the probability that $\gndx$ is simple may be almost exponentially small,
but they show that the error probability for the studied properties are
even smaller.)

Various sufficient conditions for \eqref{a2} have been given by
several authors, see
Bender and Canfield \cite{BenderC}
Bollob\'as \cite{BB80,Bollobas},
McKay \cite{McKay} and
McKay and Wormald \cite{McKayWo}.
The final result was proved in \cite{SJ195}, where, in particular, the
following was shown.
We will throughout the paper let 
\begin{equation}\label{N}
  N\=\sumd,
\end{equation}
the total number of half-edges; thus $N$ is even and the number of edges in
\gnd{} or \gndx{} is $N/2$.
(The reader that makes a detailed comparison with \cite{SJ195} should note
that the notation differs slightly.)
%; in particular, $N$ in \cite{SJ195} is the number of edges.)

\begin{theorem}[\cite{SJ195}]
  \label{T0}
Assume that $N\to\infty$. Then 
%\\
%$\liminf_\ntoo \P(\gndx\text{ is simple})>0$ if and only if
%$\sum_i d_i^2 = O(N)$.
$$
\liminf_\ntoo \P(\gndx\text{ is simple})>0
\iff
\sum_i d_i^2 = O(N).$$
\end{theorem}

Let $X_i$ be the number of loops at vertex $i$ in \gndx{}
and $X_{ij}$ the number of edges between $i$ and $j$.
Moreover, let $Y\ij\=\binom{X\ij}{2}$ be the number of pairs of parallel
edges between $i$ and $j$.
We define
\begin{equation}
  \label{z}
Z\=\sumin X_i + \sumij Y\ij;
\end{equation}
thus \gndx{} is simple $\iff Z=0$.

As shown in \cite{SJ195},  
in the case $\maxd=o(N\qq)$, it is 
not difficult to show \refT{T0} by
the method used by Bollob\'as \cite{BB80,Bollobas}, 
proving a Poisson approximation of $Z$ by the method of moments.
In general, however, $\maxd$ may be of the order $N\qq$ even when
$\sumdd=O(N)$, and in this case, $Z$ may have a 
non-Poisson asymptotic distribution. The proof in \cite{SJ195} therefore
used a more complicated method with switchings.

The purpose of this paper is to give a new proof of \refT{T0}, and of the
more precise \refT{T1} below, 
using Poisson approximations of $X_i$ and $X\ij$ to find the asymptotic
distribution of $Z$.
The new proof uses the method of moments. 
(In \cite{SJ195}, we were pessimistic about the
possibility of this; our pessimism was thus unfounded.)
The new proof presented here is conceptually simpler than the
proof in \cite{SJ195}, but it is not much shorter. The main reason for the
new proof is that it enables us to 
transfer not only results on convergence in probability but also 
some results on convergence in distribution from the random multigraph
$\gnd$ to the simple graph \gndx{} by conditioning on the existence of
specific loops or pairs of parallel edges, see
\cite{sir} for an application (which was the motivation for the present paper)
and \cite{SJ196} for an earlier example of this method in a case where
$\sumdd=o(N)$ and the results of \cite{SJ195} are enough.

We  define (with some hindsight)
\begin{align}
  \gl_i&\=\binom{d_i}2\frac1N = \frac{d_i(d_i-1)}{2N} \label{gli}
\intertext{and, for $i\neq j$,}
  \gl\ij&\= \frac{\sqrt{d_i(d_i-1)d_j(d_j-1)}}{N}, \label{glij}
\end{align}
and let $\hX_i$ and $\hX\ij$ be independent Poisson random variables with
\begin{align}\label{hX}
  \hX_i\sim\Po(\gl_i),&&&\hX\ij\sim\Po(\gl\ij).
\end{align}
In analogy with \eqref{z}, we further define
$\hY\ij\=\binom{\hX\ij}2$ and
\begin{equation}
  \label{hz}
\hZ\=\sumin \hX_i + \sumij \hY\ij
=\sumin \hX_i + \sumij \binom{\hX\ij}2.
\end{equation}

We shall show that the distribution of $Z$ is well approximated by $\hZ$,
see \refL{Ldtv}, which yields our new proof of the following estimate.
\refT{T0} is a simple corollary.

\begin{theorem}[\cite{SJ195}]\label{T1}
Assume that $N\to\infty$. Then 
  \begin{equation*}
	\begin{split}
\P\bigpar{\gndx \text{ is simple}}	
&=
\P(Z=0)
=\P(\hZ=0)+o(1)
\\&
=\exp\biggpar{-\sum_i \gl_i -\sum_{i<j}\bigpar{\gl_{ij}-\log(1+\gl_{ij})}}	  
+o(1).
	\end{split}
  \end{equation*}
\end{theorem}

As said above, our proof uses the method of moments, and most of the work
lies in showing the following estimate, proved in \refS{Spflz}. 
This is done by combinatorial
calculations that are straightforward in principle, but nevertheless rather
long.

\begin{lemma}
  \label{LZ}
Suppose that $\sumdd=O(N)$.
Then, for every fixed $m\ge1$,
\begin{equation}\label{lz}
  \E Z^m = \E \hZ^m + O\bigpar{N\qqw}.
\end{equation}
\end{lemma}

The statement means, more explicitly, that for every $C<\infty$ and $m\ge1$,
there is a constant $C'=C'(C,m)$ such that if $\sumdd\le CN$, then $|\E
Z^m-\E\hZ^m|\le C'N\qqw$. %(Other statements 

\begin{remark}\label{Rerr}
  The proof shows that the error term $O(N\qqw)$ in \eqref{lz} may be
  replaced by $O(\maxd/N)$, which always is at least as good by \eqref{dN}.
\end{remark}

In \refS{Sbi}, we give some remarks on the corresponding, but somewhat
different, result for bipartite graphs due to
\citet{BlaSta}.

\begin{ack}
  I thank Malwina Luczak for helpful comments.
\end{ack}

\section{Preliminaries}

We denote  falling factorials by $(n)_k\=n(n-1)\dotsm(n-k+1)$.

\begin{lemma}\label{Lpom}
  Let $X\in\Po(\gl)$ and let $Y\=\binom X2$. Then, for every $m\ge1$, 
$\E(Y)_m=h_m(\gl)$
for a polynomial $h_m(\gl)$ of degree $2m$.
Furthermore, $h_m$ has a double root at $0$, so $h_m(\gl)=O(|\gl|^2)$ for
$|\gl|\le1$,
and if $m\ge2$, then 
 $h_m$ has a triple root at $0$, so $h_m(\gl)=O(|\gl|^3)$ for
$|\gl|\le1$.
\end{lemma}

\begin{proof}
  $(Y)_m$ is a polynomial in $X$ of degree $2m$, and it is well-known (and
  easy to see from the moment generating function) that
  $\E X^k$ is a polynomial in $\gl$ of degree $k$ for every $k\ge0$.

Suppose that $m\ge2$.
If $X\le2$, then $Y\le1$ and thus $(Y)_m=0$.
Hence, 
\begin{equation}
  h_m(\gl)=\sum_{j=3}^\infty \lrpar{\binom j2}_m\frac{\gl^j}{j!}e^{-\gl}
=O(\gl^3)
\end{equation}
as $\gl\to0$, and thus $h_m$ has a triple root at $0$.
The same argument shows that $h_1$ has a double root at $0$; this is also
seen from the explicit formula 
\begin{equation}\label{ey}
h_1(\gl)=\E Y =\E(X(X-1)/2)=\tfrac12\gl^2.  
\end{equation}
\end{proof}

\begin{lemma}\label{Lcarl}
Let $\hZ$ be given by \eqref{hz} and assume that $\gl\ij=O(1)$.
\begin{romenumerate}[-10pt]
\item 
 For every fixed $t\ge0$,
\begin{equation}\label{lc1}
\E\exp\lrpar{t\sqrt{\hZ}}=
%  \E e^{t\sqrt{\hZ}} = 
\exp\biggpar{O\biggpar{\sum_i\gl_i+\sum_{i<j}\gl\ij^2}}.
\end{equation}
\item 
For every $C<\infty$, 
if\/ $\sum_i\gl_i+\sum_{i<j}\gl\ij^2\le C$, then
\begin{equation}\label{lc2}
  \bigpar{\E \hZ^m}^{1/m} = O(m^2),
\end{equation}
uniformly in all such $\hZ$ and $m\ge1$.
\end{romenumerate}
\end{lemma}

\begin{proof}
\pfitem{i}
  By \eqref{hz}, 
  \begin{equation}\label{cc1}
	\sqrt{\hZ}
\le \sum_i\sqrt{\hX_i}+\sum_{i<j}\sqrt{\hY\ij}
\le \sum_i{\hX_i}+\sum_{i<j}\hX\ij\ett{\hX\ij\ge2},
  \end{equation}
where the terms on the right hand side are independent.
Furthermore,
\begin{equation}\label{cc2}
  \E e^{t\hX_i}=\exp\lrpar{(e^t-1)\gl_i}=\exp\bigpar{O(\gl_i)}
\end{equation}
and, since $t$ is fixed and $\gl\ij=O(1)$,
\begin{equation}\label{cc3}
  \begin{split}
\E\exp&\lrpar{t\hX\ij\ett{\hX\ij\ge2}}
=\E e^{t\hX\ij}-\P(\hX\ij=1)(e^t-1)
\\&
=e^{(e^t-1)\gl\ij}-(e^t-1)\gl\ij e^{-\gl\ij}
\\&
=1+(e^t-1)\gl\ij\bigpar{1-e^{-\gl\ij}}+O\bigpar{\gl\ij^2}
\\&
=1+O\bigpar{\gl\ij^2}
\le\exp\bigpar{O\xpar{\gl\ij^2}}.
  \end{split}
\end{equation}
Consequently,
\eqref{lc1} follows from \eqref{cc1}--\eqref{cc3}.

\pfitem{ii}
Taking $t=1$,  (i) yields
$\exp\bigpar{\sqrt{\hZ}} \le C_1$ for some $C_1$.
Since $\exp\bigpar{\sqrt{\hZ}}\ge \hZ^m/(2m)!$, this implies
\begin{equation}
  \E \hZ^m \le (2m)!\, \E \exp\bigpar{\sqrt{\hZ}} \le C_1 (2m)^{2m},
\end{equation}
and thus $\E\bigpar{\hZ^m}^{1/m}\le 4C_1 m^2$ for all $m\ge1$.
\end{proof}

\section{Proof of \refL{LZ}}\label{Spflz}

Our proof of \refL{LZ} is rather long, although based on simple calculations,
and we will formulate a couple of intermediate steps as separate lemmas.
We begin by noting that the assumption $\sumdd=O(N)$ implies
  \begin{equation}
	\label{dN}
\maxd=O\bigpar{N\qq}
  \end{equation}
and thus, see \eqref{gli}--\eqref{glij},
\begin{equation}
  \label{glO}
\gl_i=O(1) \qquad \text{and} \qquad
\gl\ij=O(1),
\end{equation}
uniformly in all $i$ and $j$.
Furthermore, for any fixed $m\ge1$,
\begin{equation*}
  \sumi \gl_i^m = O\Bigpar{\sumi\gl_i}
= O\biggpar{\frac{\sumdd}N}=O(1).
\end{equation*}
Similarly, for any fixed $m\ge2$,
\begin{equation*}
  \sumij \gl\ij^m = O\Bigpar{\sumij\gl\ij^2}
= O\biggpar{\frac{\sum\ij d_i^2d_j^2}{N^2}}=O(1).
\end{equation*}
In particular,
\begin{equation}
  \label{judd}
\sumi\gl_i+\sumij\gl\ij^2=O(1).
\end{equation}
However, note that
there is no general bound on $\sumij\gl\ij$, as is shown by the
case of regular graphs with all $d_i=d\ge2$ and all $\binom n2$ $\gl\ij$
equal to $d(d-1)/N=(d-1)/n$, so their sum is $(d-1)(n-1)/2$.
This complicates the proof, since it forces us to obtain error estimates
involving $\gl\ij^2$. 

Let $\cH_i$ be the set of the half-edges at vertex $i$; thus $|\cH_i|=d_i$.
Further, let $\cH\=\bigcup_i\cH_i$ be the set of all half-edges.
For convenience, we order $\cH$  (by any linear order).

For $\ga,\gb\in\cH$, let $I\gab$ be the indicator that the half-edges $\ga$
and $\gb$ are joined to an edge in our random pairing. (Thus
$I\gab=I_{\gb\ga}$.)
Note that 
\begin{align}
  X_i&=\sum_{\substack{\ga,\gb\in \cH_i:\;\ga<\gb}} I\gab, \label{xi}
\\
  X\ij&=\sum_{\substack{\ga\in \cH_i,\gb\in \cH_j}} I\gab. \label{xij}
\end{align}

We have $\E I\gab=1/(N-1)$ for any distinct $\ga,\gb\in\cH$. More
generally,
\begin{equation}
  \label{eii}
\E\bigpar{I_{\ga_1\gb_1}\dotsm I_{\ga_\ell\gb_\ell}}
=
\frac{1}{(N-1)(N-3)\dotsm(N-2\ell+1)}=N^{-\ell}\ettoN
\end{equation}
for any fixed $\ell$ and any distinct half-edges
$\ga_1,\gb_1,\dots,\ga_\ell,\gb_\ell$.
Furthermore, the expectation in \eqref{eii}
vanishes if two pairs $\set{\ga_i,\gb_i}$ and
\set{\ga_j,\gb_j} have exactly one common half-edge.

We consider first, as a warm-up, $\E X_i^\ell$ for a single vertex $i$.
\begin{lemma}
  \label{LXi}
Suppose that $\sumdd=O(N)$.
Then,  for every fixed $\ell\ge1$ and all $i$,
\begin{equation}
  \label{magni}
\E X_i^\ell = \E \hX_i^\ell + O\bigpar{N\qqw\gl_i}.
\end{equation}
\end{lemma}

\begin{proof}
We may assume that $d_i\ge2$, since the case $d_i\le1$ is trivial with
$\gl_i=0$ and
$X_i=\hX_i=0$.

Since there are $\binom {d_i}2$ possible loops at $i$, 
\eqref{eii} yields
\begin{equation}\label{emm1}
  \E X_i = \binom{d_i}2\frac1{N-1}
= \gl_i\ettoN.
\end{equation}
Similarly, for any fixed $\ell\ge2$, 
there are $2^{-\ell}(d_i)_{2\ell}$ ways to select a sequence of
$\ell$ disjoint (unordered) pairs of half-edges at $i$, and thus
by \eqref{eii},
using \eqref{gli}, \eqref{dN}, \eqref{glO} and \eqref{hX},
\begin{equation}\label{emm2}
  \begin{split}
	\E (X_i)_\ell
&=
\frac{(d_i)_{2\ell}}{2^\ell N^\ell}\ettoN
=\frac{(d_i(d_1-1))^{\ell}+O\bigpar{d_i^{2\ell-1}}}{2^\ell N^\ell}\ettoN
\\&
=\gl_i^\ell\ettoN +O\Bigpar{\frac{d_i}{N}\gl_i^{\ell-1}}
\\&
=\gl_i^\ell+O\bigpar{N\qw\gl_i^{\ell}}+O\bigpar{N\qqw\gl_i^{\ell-1}}
\\&
=\gl_i^\ell+O\bigpar{N\qqw\gl_i}
\\&
=\E(\hX_i)_\ell+O\bigpar{N\qqw\gl_i}.
  \end{split}
\raisetag{\baselineskip}
\end{equation}
The conclusion \eqref{magni} now follows from \eqref{emm1}--\eqref{emm2}
and the standard relations between moments and factorial moments, 
together with \eqref{glO}.
  \end{proof}

We next consider moments of $Y\ij$, where $i\neq j$.

\begin{lemma}
  \label{LYij}
Suppose that $\sumdd=O(N)$.
Then,  for every fixed $\ell\ge1$ and all $i\neq j$,
\begin{equation}
  \label{magnij}
\E Y\ij^\ell = \E \hY\ij^\ell + O\bigpar{N\qqw\gl\ij^2}.
\end{equation}
\end{lemma}

\begin{proof}
We may assume $d_i,d_j\ge2$ since otherwise $\gl\ij=0$
and $Y\ij=\hY\ij=0$.

An unordered pair of two disjoint pairs from $\cH_i\times\cH_j$ can be
chosen in 
$\frac12d_id_j(d_i-1)(d_j-1)$ ways, and thus, by 
\eqref{eii}, \eqref{glij} and \eqref{ey},
\begin{equation}
  \label{nord}
\E Y\ij =\frac{d_id_j(d_i-1)(d_j-1)}{2(N-1)(N-3)}
=\frac{\gl\ij^2}2\ettoN
=\E\hY\ij\, \ettoN.
\end{equation}
Let $\ell\ge2$. Then $(Y\ij)_\ell$ is a sum
\begin{equation}
  \label{emma}
%(Y\ij)_\ell=
\sum_{\substack{\ga_k,\ga'_k\in\cH_i:\ga_k<\ga'_k\\\gb_k,\gb_k'\in\cH_j}}
\prod_{k=1}^\ell \lrpar{I_{\ga_k\gb_k}I_{\ga'_k\gb'_k}}
\end{equation}
where we only sum over terms such that
the $\ell$ pairs of pairs $\set{\set{\ga_k,\gb_k},\set{\ga_k',\gb_k'}}$
are distinct.

We approximate $\E(Y\ij)_\ell$ in several steps. First,
let $\tI\gab$, for  $\ga\in\cH_i$ and $\gb\in\cH_j$, 
be independent indicator variables with $\P(\tI\gab=1)=1/N$.
(In other words, $\tI\gab$ are \iid{} $\Be(1/N)$.)
Let, in analogy with \eqref{xij},
\begin{align}
  \tX\ij&\=\sum_{\substack{\ga\in \cH_i,\gb\in \cH_j}} \tI\gab, \label{txij}
\end{align}
and let $\tY\ij\=\binom{\tX\ij}2$. Then, $(\tY\ij)_\ell$ is a sum similar to
\eqref{emma}, with $I\gab$ replaced by $\tI\gab$.
Note that \eqref{emma} is a sum of terms that are products of $2\ell$
indicators; however, there may be repetitions among the indicators, so each
term is a product of $r$ distinct indicators where $r\le2\ell$. 
Since we assume $\ell\ge2$, and the pairs 
$\set{\set{\ga_k,\gb_k},\set{\ga_k',\gb_k'}}$ are distinct, $r\ge3$ for each
term. 

Taking expectations and using \eqref{eii}, we see that the terms in
\eqref{emma} where all occuring pairs $\set{\ga_k,\gb_k}$ are distinct yield
the same contributions to $\E(Y\ij)_\ell$ and $\E(\tY\ij)_\ell$, apart from
a factor $\ettoN$.

However, there are also terms containing factors $I\gab$ and $I_{\ga'\gb'}$
where $\ga=\ga'$ or $\gb=\gb'$ (but not both). Such terms vanish identically
for $(Y\ij)_\ell$, but the corresponding terms for $(\tY\ij)_\ell$ do not.
The number of such terms for a given $r\le2\ell$ is
$O(d_i^{r-1}d_j^r+d_i^rd_j^{r-1})$ and thus their contribution to
$\E(\tY\ij)_\ell$ is, using \eqref{eii} and \eqref{dN},
\begin{equation}
O\parfrac{d_i^{r-1}d_j^r+d_i^rd_j^{r-1}}{N^r}
=
O\lrpar{\frac{d_j+d_i}{N}\gl\ij^{r-1}}
=
O\lrpar{{N}\qqw\gl\ij^{r-1}}.
\end{equation}
Summing over $3\le r\le2\ell$, this yields, using \eqref{glO},
a total contribution $O\lrpar{{N}\qqw\gl\ij^{2}}$. 
Consequently, we have
\begin{equation}
  \label{a8}
\E(Y\ij)_\ell
=\E(\tY\ij)_\ell\ettoN + 
O\lrpar{{N}\qqw\gl\ij^{2}}.
\end{equation}

Next, replace the \iid{} indicators $\tI\gab$ by \iid{} Poisson variables
% for  $\ga\in\cH_i$ and $\gb\in\cH_j$, 
$\uI\gab\sim\Po(1/N)$ with the same mean, and let,
in analogy with \eqref{xij} and \eqref{txij},
\begin{align}
  \uX\ij&\=\sum_{\substack{\ga\in \cH_i,\gb\in \cH_j}} \uI\gab
\sim\Po\parfrac{d_id_j}{N}, \label{uxij}
\end{align}
and let $\uY\ij\=\binom{\uX\ij}2$. Then, $(\uY\ij)_\ell$ can be expanded as
a sum similar to 
\eqref{emma}, with $I\gab$ replaced by $\uI\gab$. We take the expectation
and note that the only difference from $\E(\tY\ij)_\ell$ is for terms where
some $\uI\gab$ is repeated. We have, for any fixed $k\ge1$,
\begin{equation}
  \E\uI\gab^k=\frac1N+O(N\qww)
=\frac1N\ettoN,
\end{equation}
while
\begin{equation}
  \E \tI\gab^k=\E\tI\gab=\frac1N.
\end{equation}
Hence, for each term, the difference, if any, is by a factor $\ettoNx$,
and thus
\begin{equation}
  \label{ba}
\E(\uY\ij)_\ell=\E(\tY\ij)_\ell\ettoN.
\end{equation}

Note that we here use $\uY\ij=\binom{\uX\ij}2$, where
$\uX\ij\sim\Po(d_id_j/N)$ has a mean 
$\ugl\ij\=d_id_j/N$
that differs from $\E\hX\ij=\gl\ij$
given by \eqref{glij}.
We have
\begin{equation}\label{a9}
  \ugl\ij \ge \gl\ij \ge \frac{(d_i-1)(d_j-1)}{N}>\ugl\ij -  \frac{d_i+d_j}{N}.
\end{equation}
We use \refL{Lpom} and note that the lemma implies that $h_\ell'(\gl)=O(\gl^2)$
for each $\ell\ge2$ and $\gl=O(1)$. Hence, by \eqref{a9} and
\eqref{dN}--\eqref{glO}, 
\begin{equation}\label{bax}
  \begin{split}
\E(\uY\ij)_\ell-\E(\hY\ij)_\ell
&=h_\ell(\ugl\ij)-h_\ell(\gl\ij)
=O\bigpar{\ugl\ij^2(\ugl\ij-\gl\ij)}	
\\&
=O\Bigpar{\frac{d_i+d_j}{N}\ugl\ij^2}
=O\Bigpar{N\qqw\gl\ij^2}.
  \end{split}
\end{equation}
Finally, \eqref{a8}, \eqref{ba} and \eqref{bax} yield, for each $\ell\ge2$,
\begin{equation}
  \label{sw}
\E(Y\ij)_\ell  = \E(\hY\ij)_\ell \ettoN+O\Bigpar{N\qqw\gl\ij^2}.
\end{equation}
By \eqref{nord}, this holds for $\ell=1$ too.
By \eqref{glO} and \refL{Lpom}, for each $\ell\ge1$,
$\E(\hY\ij)_\ell=O(\gl\ij^2)$, and thus \eqref{sw} can be written
\begin{equation}
  \label{sw0}
\E(Y\ij)_\ell  = \E(\hY\ij)_\ell +O\Bigpar{N\qqw\gl\ij^2},
\end{equation}
for each fixed $\ell\ge1$.
The conclusion now follows, as in \refL{LXi}, by the relations between
moments and factorial moments, again using the bound \eqref{glO}.
\end{proof}

In particular, note that Lemmas \refand{LXi}{LYij} 
together with \refL{Lpom} and \eqref{glO} imply the bounds, for every
fixed $\ell\ge1$, 
\begin{align}
  \label{erikai}
\E X_i^\ell+ \E \hX_i^\ell 
&=O\bigpar{\gl_i+\gl_i^\ell+N\qqw\gl_i}=O(\gl_i).
\\
  \label{erikaij}
\E Y\ij^\ell+ \E \hY\ij^\ell 
&=O\bigpar{\gl\ij^2+\gl\ij^{2\ell}+N\qqw\gl\ij^2}=O(\gl\ij^2).
\end{align}

\begin{proof}[Proof of \refL{LZ}]

We uncouple the  terms in \eqref{z} by letting $(I\gab\xij)\gagb$ be
independent copies of $(I\gab)\gagb$, for $1\le i,j\le n$, and defining, in
analogy with \eqref{xi}--\eqref{xij} and \eqref{z},
\begin{align}
\uc X_i&\=\sum_{\substack{\ga,\gb\in \cH_i:\;\ga<\gb}} I\gab\xii, \label{ucxi}
\\
\uc X\ij&\=\sum_{\substack{\ga\in \cH_i,\gb\in \cH_j}} I\gab\xij, \label{ucxij}
\\
\uc Y\ij & \=\binom{\uc X\ij}{2},
\\
\uc Z&\=\sumin \uc X_i + \sumij \uc Y\ij. \label{ucz}
\end{align}
Note that the summands in \eqref{ucxi} are not independent; they have the same
structure as $(I_{\ga\gb})_{\ga,\gb\in\cH_i}$ and thus $\uc X_i\eqd X_i$, 
and similarly
$X\ij\eqd X\ij$. However, different sums $\uc X_i$ and $\uc X\ij$ are
independent (unlike $X_i$ and $X\ij$).

We begin by comparing $\E\uc Z^m$ and $\E\hZ^m$. Since the terms in
\eqref{ucz} are independent, the moment $\E \uc Z^m$ can be written as a
certain polynomial $g_m\bigpar{\E X_i^\ell,\E Y\ij^\ell:i,j\in\nn,\, \ell\le m}$
in the moments
$\E(\uc X_i)^\ell=\E X_i^\ell$
and
$\E(\uc X\ij)^\ell=\E X\ij^\ell$ for $1\le \ell\le m$ and $i,j\in\nn$.

By \eqref{hz}, $\E\hZ^m$ can be expressed in the same way as 
$g_m\bigpar{\E \hX_i^\ell,\E \hY\ij^\ell:i,j\in\nn,\, \ell\le m}$ for the same
polynomial $g_m$. It follows that
\begin{equation}\label{jui}
  \E\uc Z^m - \E\hZ^m 
= \sum_{\ell=1}^m \sum_i\bigpar{\E X_i^\ell-\E	\hX_i^\ell}R_{\ell i}
+
\sum_{\ell=1}^m\sum_{i<j}\bigpar{\E Y\ij^\ell-\E\hY\ij^\ell}R_{\ell ij}
\end{equation}
for some polynomials $R_{\ell i}$ and $R_{\ell ij}$ in the moments 
$\E X_i^k$, $\E\hX_i^k$, $\E Y\ij^k$, $\E\hY\ij^k$, for $k\le m$;
it is easily seen from \eqref{erikai}--\eqref{erikaij} and \eqref{judd} that
\begin{equation}
  R_{\ell i},R_{\ell ij} 
= O\lrpar{\sum_{\nu=1}^m\Bigpar{\sum_i\gl_i+\sum_{i<j}\gl\ij^2}^\nu}
=O(1)
\end{equation}
uniformly in $i,j\in\nn$ and $\ell\le m$.
Hence, \eqref{jui} yields, together with \eqref{magni}, \eqref{magnij}
and \eqref{judd},
\begin{equation}\label{zuc1}
  \E\uc Z^m - \E\hZ^m 
= O\biggpar{ \sum_i N\qqw\gl_i+\sum_{i<j} N\qqw\gl\ij^2}
=O\bigpar{N\qqw}.
\end{equation}

It remains to compare $\E Z^m$ and $\E\uc Z^m$.
By \eqref{z} and \eqref{xi}--\eqref{xij}, $Z^m$ can be expanded as a sum 
of certain products
\begin{equation}\label{iabab}
  I_{\ga_1\gb_1}\dotsm I_{\ga_\ell\gb_\ell}
\end{equation}
where $1\le \ell\le 2m$ and we may assume that the pairs
\set{\ga_1,\gb_1},\dots,\set{\ga_\ell,\gb_\ell} are distinct. 
(Some products \eqref{iabab} may be repeated in $Z^m$, but only $O(1)$ times.)
Furthermore,
by \eqref{ucxi}--\eqref{ucz},
$\uc Z^m$ is the sum of the corresponding products
\begin{equation}\label{uciabab}
\uc I_{\ga_1\gb_1}\dotsm \uc I_{\ga_\ell\gb_\ell},
\end{equation}
where $\uc I\gab\=I\gab\xij$ when $\ga\in\cH_i$ and $\gb\in\cH_j$.

We say that a product \eqref{iabab} or \eqref{uciabab}
is \emph{bad} if 
it contains two factors $I_{\ga_\nu\gb_\nu}$ and $I_{\ga_\mu\gb_\mu}$
such that the pairs $\set{\ga_\nu\gb_\nu}$ and $\set{\ga_\mu\gb_\mu}$
contain a common index, say $\ga_\nu=\ga_\mu$, and furthermore the two
remaining indices, $\gb_\nu$ and $\gb_\mu$, say, are half-edges belonging to
different vertices, \ie, $\gb_\nu\in\cH_i$ and $\gb_\mu\in\cH_j$ with $i\neq
j$. Otherwise we say that the product is \emph{good}.
(Note that a good product may contain factors
$I_{\ga_\nu\gb_\nu}$ and $I_{\ga_\mu\gb_\mu}$
with $\ga_\nu=\ga_\mu$ as long as $\gb_\nu$ and $\gb_\mu$ belong to the
same vertex.)
It follows from \eqref{eii} that for each good product, the corresponding
contributions to $\E Z^m$ and $\E\uc Z^m$ differ only by a factor $\ettoN$.
For a bad product, however, the contribution to $\E Z^m$ is 0.
We thus have to estimate the contribution to $\E \uc Z^m$ of the bad products.

We define the \emph{support} of a product \eqref{uciabab} as the multigraph
with vertex set $\nn$ and edge set \set{\ga_\nu\gb_\nu:1\le\nu\le \ell},
\ie, the multigraph obtained by forming edges from the pairs of half-edges
appearing as indices in the product.
If $F$ is the support of \eqref{uciabab}, then $F$ thus has $\ell$ edges
(possibly including loops). Furthermore, 
it follows from \eqref{ucxi}--\eqref{ucz} that
every edge in $F$ that is not a
loop has at least one edge parallel to it. Hence, a vertex $i$
in $F$ with non-zero degree has degree at least 2. 
In other words,  if we denote the vertex
degrees in $F$ by $\gd_1,\dots,\gd_n$, then $\gd_i=0$ or $\gd_i\ge2$.
Moreover, if
\eqref{uciabab} is bad with, say, $\ga_\nu=\ga_\mu\in\cH_i$, then there are
edges in $F$ from $i$ to at least two vertices $j$ and $k$ (one of which
may equal $i$), and thus the degree $\gd_i\ge4$.

Let $F$ be a multigraph with vertex set $\nn$ and $\ell$ edges, and denote
again its vertex degrees by $\gd_1,\dots,\gd_n$. Thus $\sum_i\gd_i=2\ell$.
Let $S_F$ be the contribution to $\E\uc Z^m$ from bad products \eqref{iabab}
with support $F$. A bad product has some half-edge repeated, and if this
belongs to $\cH_i$, there are $O(d_i^{\gd_i-1}\prod_{j\neq i}d_j^{\gd_j})$
choices for the product. Furthermore, as just shown,
this can only occur for $i$ with $\gd_i\ge4$. Since each product yields a
contribution $O(N^{-\ell})$ by \eqref{eii}, we have,
using
$2\ell=\sum_i\gd_i$ and \eqref{dN} together with the fact that
$\gd_j\neq1$,
\begin{equation}\label{sf}
  \begin{split}
S_F&
=O\lrpar{N^{-\ell}\sum_{i:\gd_i\ge4}d_i^{\gd_i-1}\prod_{j\neq i}d_j^{\gd_j}}
\\&
=O\lrpar{N^{-1/2}\sum_{i:\gd_i\ge4}\parfrac{d_i}{N\qq}^{\gd_i-1}
 \prod_{j\neq i}\parfrac{d_j}{N\qq}^{\gd_j}}
\\&
=O\lrpar{N^{-1/2} \prod_{j:\gd_j>0}\parfrac{d_j}{N\qq}^{2}}.
  \end{split}
\end{equation}

Summing over all possible $F$, and recalling that $\ell\le 2m$,
it follows that the total contribution to $\E\uc Z^m$ from bad products is 
\begin{equation}\label{sumfs}
\sum_F S_F
=O\lrpar{N^{-1/2}\sum_F \prod_{j:\gd_j>0}\frac{d_j^2}{N}}
.
\end{equation}
For each support $F$, the set $\set{j:\gd_j>0}=\set{j:\gd_j\ge2}$ has size
at most $\ell\le 2m$, and for
each choice of this set, there are $O(1)$ possible $F$.
Hence,
\begin{equation*}
  \begin{split}
\sum_F \prod_{j:\gd_j>0}\frac{d_j^2}{N}
=O\lrpar{\sum_{k=1}^{2m}\sum_{j_1<\dots<j_k}\prod_{i=1}^k 
 \frac{d_{j_i}^2}{N}}
%\\&
=O\lrpar{\sum_{k=1}^{2m} \biggpar{\sum_{j=1}^n\frac{d_j^2}{N}}^k}
=O(1),
  \end{split}
\end{equation*}
and \eqref{sumfs} yields
\begin{equation}
\sum_F S_F =O\bigpar{N^{-1/2}}.
\end{equation}

Summarizing, the argument above yields
\begin{equation}\label{zuc2}
  \E Z^m = \E \uc Z^m\ettoN + O\bigpar{N\qqw}
= \E \uc Z^m + O\bigpar{N\qqw},
\end{equation}
since $\E \uc Z^m=O(1)$, \eg{} by \eqref{zuc1} and \refL{Lcarl}. (Or by
arguing similarly as above, summing over supports.)

The lemma follows from \eqref{zuc1} and \eqref{zuc2}.
\end{proof}

\section{Proof of theorems \ref{T0} and \ref{T1}}\label{Spf}

We first assume that $\sumdd=O(N)$ and prove the following, more precise,
statement. 

\begin{lemma}
  \label{Ldtv}
Suppose that $\sumdd=O(N)$ and \Ntoo. Then
  $\dtv(Z,\hZ)\to0$.
\end{lemma}

\begin{proof}
  We argue by contradiction. Suppose that the conclusion fails. Then there
  is a sequence of sequences $(d_i)_1^n$ %=(d_i\nnu)_1^{n(\nu)}$, $\nu\ge1$, 
with
  $N\=\sumd\to\infty$ and $\sumdd=O(N)$, but $\dtv(Z,\hZ)\ge\eps$ for some
  $\eps>0$. %(For convenience, we omit the index $\nu$ from the notation.)
Note that then \eqref{judd}  holds.

By \refL{Lcarl}(ii) and \eqref{judd}, $\E\hZ^m=O(1)$ for each $m$.
In particular, the  sequence $\hZ$ is tight, and by  selecting a subsequence
we may assume that $\hZ\dto \zoo$ for some random variable $\zoo$.
Furthermore, the estimate $\E\hZ^m=O(1)$ for each $m$ implies that
$\hZ^m$ is uniformly integrable for each $m\ge1$, and thus
\begin{equation}\label{hzoom}
  \E \hZ^m\to\zoo^m,
\end{equation}
see \eg{} 
\cite[Theorems 5.4.2 and 5.5.9]{Gut}.
By \refL{LZ}, we thus also have
\begin{equation}\label{zoom}
  \E Z^m\to\zoo^m
\end{equation}
for each $m\ge1$.
Furthermore, by \eqref{lc2} and \eqref{hzoom},
\begin{equation}\label{zoom2}
\bigpar{ \E \zoo^m}^{1/m}=O(m^2).
\end{equation}
We can now apply the method of moments and conclude from \eqref{zoom}
that $Z\dto \zoo$. We justify the use by the method of moments by \eqref{zoom2},
which implies that
\begin{equation}
  \label{carl+}
\sum_m \bigpar{ \E \zoo^m}^{-1/2m}=\infty; 
\end{equation}
since $\zoo\ge0$, this weaker form of the usual Carleman criterion shows
that the 
distribution of $\zoo$ is determined (among all distributions on $\ooo$)
by its moments, and thus (since also $Z\ge0$)
the method of moment applies, see \eg{} \cite[Section 4.10]{Gut}.

Hence $\hZ\dto\zoo$ and $Z\dto\zoo$, and thus
\begin{equation}
  \dtv(Z,\hZ)\le\dtv(Z,\zoo)+\dtv(\hZ,\zoo)\to0,
\end{equation}
a contradiction which proves the lemma.
\end{proof}

\begin{remark}\label{Rtorsten}
  Note that $\E e^{t\hY\ij}=\infty$ for every $t>0$ when $\gl\ij>0$. Hence,
$\hZ$ does not have a finite \mgf. Similarly, it is possible that $\E
e^{t\zoo}=\infty$;
consider, for example, the case $d_1=d_2\sim N\qq$ when $\gl_{12}\to1$ and
$\zoo\ge\binom{\hX}2$ with $\hX\sim\Po(1)$.
In this case, furthermore, %$\bigpar{\E \zoo^m}^{1/m}>cm^2/\log^2 m$
by Minkowski's inequality,
\begin{equation}
  \begin{split}
\bigpar{\E\zoo^m}^{1/m}
&\ge\frac12\bigpar{\E(\hX^2-\hX)^m}^{1/m}
\ge\frac12\bigpar{\E \hX^{2m}}^{1/m} - \frac12\bigpar{\E \hX^{m}}^{1/m}	
\\&\sim
\frac12\Bigparfrac{2m}{e\log m}^2
=\frac{2m^2}{e^2\log^2m}
  \end{split}
\end{equation}
using simple estimates
for the 
moments $\E\hX^m$ when $\hX\sim\Po(1)$, which are the
Bell numbers.
(Or by more precise asymptotics in \eg{} \cite[Proposition VIII.3]{Flajolet} and
\cite[\S26.7]{NIST}.) 
Hence, in this case,
$\sum_m \bigpar{\E \zoo^m}^{-1/m}\allowbreak<\infty$; in other words,
$\zoo$ does not satisfy
the usual Carleman criterion 
$\sum_m \bigpar{\E \zoo^m}^{-1/m}\allowbreak=\infty$ for the distribution of
to be determined by its moments.
However, since we here deal with non-negative random variables, we can use
the weaker condition \eqref{carl+}. 
(This weaker version is well-known, and follows from the standard version by
considering the square root $\pm\zoo$ with random sign, independent of
$\zoo$.
Alternatively, we may observe that \eqref{zoom} implies $\E(\pm \sqrt
Z)^k\to\E(\pm\sqrt\zoo)^k$ for all $k\ge0$, where the moments trivially
vanish when $k$ is odd; since $\pm\sqrt{\zoo}$ has a finite \mgf{} by
\eqref{lc1} and Fatou's lemma, the usual sufficient condition for the method
of moments yields $\pm\sqrt{Z}\dto\pm\sqrt{\zoo}$, and thus $Z\dto\zoo$.)
\end{remark}

\begin{proof}[Proof of Theorems \ref{T0} and \ref{T1}]
In the case $\sumdd=O(N)$, \refT{T1} follows from \refL{Ldtv}, since
\begin{equation}\label{asa}
  \begin{split}
\P(\hZ=0)
&=\P\bigpar{\hX_i= \hY\ij=0 \text{ for all }i,j} 
=\prod_i\P(\hX_i=0)\prod_{i<j}\P(\hX\ij\le1)
\\
&=\prod_i e^{-\gl_i}\prod_{i<j}(1+\gl\ij)e^{-\gl\ij}.
  \end{split}
\raisetag{1.4\baselineskip}
\end{equation}
Furthermore, 
$\gl\ij-\log(1+\gl\ij)=O(\gl\ij^2)$, so it follows from this and %\refT{T1} and
\eqref{judd} that 
$\liminf_\ntoo\P\bigpar{\gndx \text{ is simple}}	>0$, verifying \refT{T0}
in this case.

It remains (by considering subsequences) only to consider the case when
$\sumdd/N\to\infty$. Since then
\begin{equation}\label{asb}
  \sum_i\gl_i=\frac{\sumdd-\sumd}{2N}
=
\frac{\sumdd}{2N}-\frac12\to\infty, 
\end{equation}
it follows from \eqref{asa} that $\P(\hZ=0)\to0$, and it remains to show
that $\P(Z=0)\to0$. We do this by the method used in \cite{SJ195}  for this
case. We fix $A>1$ and split vertices by replacing some $d_j$ by $d_j-1$
and a new vertex $n+1$ with $d_{n+1}=1$, repeating until the new degree
sequence, $(\hd_i)_1^{\hn}$ say, satisfies $\sumi\hd_i^2\le AN$. (Note that
the number $N$ of half-edges is unchanged.)
Then, as $N\to\infty$, see \cite{SJ195} for details,
$\sumi\hd_i^2\sim AN$ and, denoting the new random multigraph by $\hG$ and
using \refL{Ldtv} together with \eqref{asa} and \eqref{asb} on $\hG$,
\begin{equation*}
  \begin{split}
\P\bigpar{\gnd\text{ is simple}}	
&\le
\P\bigpar{\hG\text{ is simple}}	
\le 
\exp\Bigpar{-\sum_i\frac{\hd_i(\hd_i-1)}{2N}}+o(1)
\\&
=\exp\Bigpar{-\frac{\sum_i\hd_i^2}{2N}+\frac12}+o(1)
%\\&
\to \exp\Bigpar{-\frac{A-1}2}.
  \end{split}
\end{equation*}
Since $A$ is arbitrary, it follows that 
$\P\bigpar{\gnd\text{ is simple}}=\P(Z=0)\to0$ in this case, which completes
the proof.	   
\end{proof}

\section{Bipartite graphs}\label{Sbi}

A similar result for bipartite graphs has been proved by
\citet{BlaSta}; see 
\eg{} 
\cite{BBK},
\cite{McKayrect},
\cite{GreenhillMW} for earlier results.
(These results are often stated in an equivalent form about $0$-$1$
matrices.)
We suppose we are given degree sequences $(s_i)_1^{n'}$ and $(t_j)_1^{n''}$
for the two parts, with $N:=\sum_i s_i = \sum_j t_j$, and consider a random 
bipartite simple graph $\ggb$ with these degree sequences as well as the
corresponding random bipartite multigraph $\ggx=\ggbx$
constructed by the configuration model. (These have $N$ edges.)
We order the two degree sequences in decreasing order as 
$s_{(1)}\ge\dots\ge s_{(n')}$ and $t_{(1)}\ge\dots\ge t_{(n'')}$, and let 
$s:=s_{(1)}=\max_i s_i$ and $t:=t_{(1)}=\max_j t_j$.
Label the vertices in the two parts $v_1,\dots,v_{n'}$ and
$w_1,\dots,w_{n''}$, in order of decreasing degrees; thus $v_i$ [$w_j$]
has degree $s_{(i)}$ [$t_{(j)}$].

\begin{theorem}[\citet{BlaSta}]\label{Tb}
Assume that $N\to\infty$. Then 
$\liminf_\ntoo \P\bigpar{\ggbx\text{ is simple}}>0$ if and only if the
following two conditions hold:
\begin{romenumerate}
\item 
  \begin{equation}\label{tb1}
    \sum_i\sum_j s_i(s_i-1)t_j(t_j-1) = O(N^2).
  \end{equation}
\item 
For any fixed $m\ge1$,
\begin{align}
  \sum_{i=\min\set{t,m}}^{n'} s_{(i)} =\gO(N), \label{tb2s}
\\
\sum_{j=\min\set{s,m}}^{n''} t_{(j)} =\gO(N). \label{tb2t}
\end{align}
\end{romenumerate}
\end{theorem}
(We have reformulated (ii) from \cite{BlaSta} somewhat. Recall that $x=\gO(N)$
means that $\liminf x/N>0$.)

\begin{remark}
Here (i) corresponds to the condition $\sumdd=O(N)$ in \refT{T0}, while (ii)
is an additional complication.
Note that if $s=o(N)$ then \eqref{tb2s} holds, because the sum is $\ge
N-(m-1)s$; similarly, if $t=o(N)$ then \eqref{tb2t} holds.
Hence (ii) is  satisfied, and (i) is sufficient,  unless for some
subsequence either $s=\gO(N)$ or $t=\gO(N)$. Note also that both these
cannot occur when \eqref{tb1} holds; in fact, if $s=\gO(N)$, then
\eqref{tb1} implies $\sum_j t_j(t_j-1)=O(1)$ and thus $t=O(1)$.  
%and only $O(1)$ of the $t_j$ are greater than 1.
On the other hand, in such cases, (i) is not enough, as pointed out by
\citet{BlaSta}. For example, if $s_1=N-o(N)$, $t_1=2$ and $t_j=1$ for $j\ge2$,
then (i) holds but \eqref{tb2s} fails for $m=2$. Indeed, in this example,
there is \whp{} (\ie, with probability $1-o(1)$) a double edge $v_1w_1$, and
thus $\ggx$ is \whp{} not simple.
\end{remark}

We can prove \refT{Tb} too by the methods of this paper. 
(The proof by \citet{BlaSta} is different.)
There are no loops,
and thus no $X_i$, but we define $X\ij$ and $Y\ij$ as above (with the
original labelling) and let
$Z\=\sum_{i=1}^{n'} \sum_{j=1}^{n''} Y\ij$. 
Similarly, we define, for $i\in [n']$ and $j\in[n'']$,
\begin{align}
  \gl\ij&\= \frac{\sqrt{s_i(s_i-1)t_j(t_j-1)}}{N}, \label{bglij}
\end{align}
let  $\hX\ij\sim\Po(\gl\ij)$ and 
$\hY\ij\=\binom{\hX\ij}2$ be as above and let
$\hZ\=\sum_{i=1}^{n'} \sum_{j=1}^{n''} \hY\ij$. 
Note that \eqref{tb1} is $\sum_{i,j}\gl\ij^2=O(1)$.

\begin{theorem}\label{Tb1}
  Assume that $N\to\infty$ and that $s,t=o(N)$.
Then
  \begin{equation*}
	\begin{split}
\P\bigpar{\ggbx \text{ is simple}}	
&=
\P(Z=0)
=\P(\hZ=0)+o(1)
\\&
=\exp\biggpar{\sum_{i,j}\bigpar{\gl_{ij}-\log(1+\gl_{ij})}}	  
+o(1).
	\end{split}
  \end{equation*}
\end{theorem}

\begin{proof}[Proof (sketch)]
  This is proved as \refT{T1}, using analogues of Lemmas \ref{LZ} and
  \ref{Ldtv}, with only minor differences. Instead of \eqref{dN} we use the
  assumption $s,t=o(N)$, which leads to error terms of the order $O((s+t)/N)$,
\cf{} \refR{Rerr}.
Furthermore, \eqref{sf} has to be modifed.
Say that the vertex with a repeated half-edge is bad, and suppose that the
bad vertex is in the first part. Let the non-zero vertex degrees in $F$ be
$a_1,a_2,\dots$ in the first part and $b_1,b_2,\dots $ in the second part,
in any order with the bad vertex having degree $a_1$. 
Thus $\sum_\nu a_\nu=\sum_\mu b_\mu=\ell$.
The contribution from all $F$ with
given $(a_\nu)$ and $(b_\mu)$ is, using \Holder's inequality and \eqref{tb1},
\begin{equation*}
  \begin{split}
&O\biggpar{N^{-\ell}\sum_{i:s_i\ge2}s_i^{a_1-1}
\prod_{\nu\ge2}\Bigpar{\sum_{i:s_i\ge2} s_i^{a_\nu}}
\prod_{\mu\ge1}\Bigpar{\sum_{j:t_j\ge2} t_j^{b_\mu}}}
\\&\qquad
=
O\biggpar{N^{-\ell}
\Bigpar{\sum_{i:s_i\ge2}s_i^2}^{(a_1-1)/2+\sum_{\nu\ge2}a_\nu/2}
\Bigpar{\sum_{j:t_j\ge2} t_j^2}^{\sum_\mu b_\mu/2}
}
\\&\qquad
=
O\biggpar{N^{-\ell}
\Bigpar{\sum_{i}s_i(s_i-1)}^{(\ell-1)/2}
\Bigpar{\sum_{j}t_j(t_j-1)}^{\ell/2}
}
\\&\qquad
=
O\biggpar{N^{-1}
\Bigpar{\sum_{j}t_j(t_j-1)}^{1/2}
}
=O\Bigpar{t\qq/N\qq}.
  \end{split}
\end{equation*}
Summing over the finitely many
$(a_\nu)$ and $(b_\mu)$, and adding the  case with the bad vertex
in the second part,
we obtain $O\bigpar{(s+t)\qq/N\qq}=o(1)$.
\end{proof}

\begin{proof}[Proof of \refT{Tb}]
  The case $s,t=o(N)$ (when (ii) is automatic)
follows from \refT{Tb1}.

By considering subsequences, and symmetry, it remains only to consider the
case $s=\gO(N)$. It is easy to see that (i)  is necessary in this case too
%(for example by splitting some vertices), 
so we may assume (i). As said
above, this implies $t=O(1)$, and furthermore, that only $O(1)$ degrees
$t_j$ are $>1$. By taking a further subsequence, we may assume that $t$ is
constant. 
Then \eqref{tb2t} always holds, and 
it suffices to consider the case $m=t$ in \eqref{tb2s}, \ie,
\begin{align}
  \sum_{i={t}}^{n'} s_{(i)} =\gO(N). \label{tb2st}
\end{align} 

If \eqref{tb2st} does not hold, then (at least for a subsequence),
\whp{} 
$ \sum_{i={t}}^{n'} s_{(i)} =o(N)$, and then \whp{} the $t$ edges from $w_1$
go only to \set{v_i:i<t}, so by the pigeonhole principle, there is a double
  edge. 

Conversely, if \eqref{tb2st} holds, it is easy to see that if we first match
the half-edges from $w_1$, $w_2$, \dots, in this order, there is 
(for large $n$) for each
half-edge a probability at least $\eps$ for some $\eps>0$ 
to not create a double edge; since
there are only $O(1)$ such vertices with $t_j>1$, it follows that 
$\P\bigpar{\ggx\text{ is simple}}$ is bounded below.
\end{proof}

\newcommand\AAP{\emph{Adv. Appl. Probab.} }
\newcommand\JAP{\emph{J. Appl. Probab.} }
\newcommand\JAMS{\emph{J. \AMS} }
\newcommand\MAMS{\emph{Memoirs \AMS} }
\newcommand\PAMS{\emph{Proc. \AMS} }
\newcommand\TAMS{\emph{Trans. \AMS} }
\newcommand\AnnMS{\emph{Ann. Math. Statist.} }
\newcommand\AnnPr{\emph{Ann. Probab.} }
\newcommand\CPC{\emph{Combin. Probab. Comput.} }
\newcommand\JMAA{\emph{J. Math. Anal. Appl.} }
\newcommand\RSA{\emph{Random Structures Algorithms} }
\newcommand\ZW{\emph{Z. Wahrsch. Verw. Gebiete} }
\newcommand\DMTCS{\jour{Discr. Math. Theor. Comput. Sci.} }

\newcommand\AMS{Amer. Math. Soc.}
\newcommand\Springer{Springer-Verlag}
\newcommand\Wiley{Wiley}

\newcommand\vol{\textbf}
\newcommand\jour{\emph}
\newcommand\book{\emph}
\newcommand\inbook{\emph}
\def\no#1#2,{\unskip#2, no. #1,} %(typeset after year) 
\newcommand\toappear{\unskip, to appear}

\newcommand\urlsvante{\url{http://www.math.uu.se/~svante/papers/}}
\newcommand\arxiv[1]{\url{arXiv:#1.}}
\newcommand\arXiv{\arxiv}

\def\nobibitem#1\par{}

\end{document}